\documentclass[10pt]{article}
\usepackage[noadjust]{cite}

\usepackage{tikz}
\usepackage{amsmath,amsthm,amsfonts,amssymb,latexsym,graphicx}
\usepackage{etoolbox}
\usepackage{algorithm}
\usepackage[noend]{algpseudocode}
\usepackage[pdfpagemode=UseNone,pdfstartview=FitH]{hyperref}

\allowdisplaybreaks[4]

\newtheorem{theorem}{Theorem}
\newtheorem{lemma}[theorem]{Lemma}
\newtheorem{corollary}[theorem]{Corollary}

\theoremstyle{definition}

\newtheorem{remark}{Remark}

\newcommand{\E}{\mathbb{E}}
\renewcommand{\P}{\mathbb{P}}

\newcommand*{\dd}{\,\mathrm{d}}

\title{Conformal e-prediction in the presence of confounding}
\author{Vladimir Vovk and Ruodu Wang}

\begin{document}
\maketitle
\begin{abstract}
  This note extends conformal e-prediction to cover the case
  where there is observed confounding between the random object $X$
  and its label $Y$.
  We consider both the case where the observed data is IID
  and a case where some dependence between observations is permitted.

   The version of this note at \url{http://alrw.net} (Working Paper 46)
   is updated most often.
\end{abstract}

\section{Introduction}

Conformal prediction in its basic form is only applicable to IID sequences of observations.
In causal inference,
including Pearl's \cite[Chap.~3]{Pearl:2009} do calculus,
we typically observe IID data but then would like to predict
results of interventions into stable stochastic mechanisms generating the data.
In this note we apply conformal e-prediction to this prediction problem
in order to obtain finite-sample guarantees of validity.
Formally, however, the specific approach that we take in this note
goes beyond conformal prediction and is closer to ``randomness prediction''
(as defined in, e.g., \cite{Vovk:2026ALT}).

\begin{figure}
  \begin{center}
    \begin{tikzpicture}
      \node (v0) at (0,0) {X};
      \node (v1) at (4,0) {Y};
      \node (v2) at (2,1) {Z};
      \draw [->] (v2) edge (v0);
      \draw [->] (v0) edge (v1);
      \draw [->] (v2) edge (v1);
    \end{tikzpicture}
  \end{center}
  \caption{The main causal graph of this note}
  \label{fig:XYZ}
\end{figure}
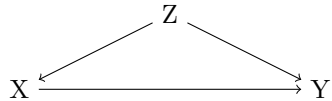

The main question asked in this note is about the simplest setting of causal inference.
We are interested in the causal effect of a random variable $X$
on a random variable $Y$ with a confounder $Z$.
(See Figure~\ref{fig:XYZ}.)
Namely, after setting $X:=x$, we would like to say something about $Y$,
e.g., to output a prediction region for it.
The available data comes from an observational study.
We consider two settings:
in Sect.~\ref{sec:IID} we start from the standard IID one,
while in Sect.~\ref{sec:game} we allow $X$ to be chosen by a non-trivial strategy,
as in \cite{Vovk:1996Gam} and \cite[Sect.~3.6.1, third extension]{Pearl:2009}.
Our mathematical results in Sect.~\ref{sec:IID} are simple and likely to be known.

For a positive integer $N$, we set $[N]:=\{1,\dots,N\}$.

\section{The IID setting}
\label{sec:IID}

Let $P$ be a positive probability measure on $\mathbf{X}\times\mathbf{Y}\times\mathbf{Z}$
generating random variables $(X,Y,Z)$;
for simplicity we assume that $\mathbf{X}\times\mathbf{Y}\times\mathbf{Z}$ is finite
(and equipped with the discrete $\sigma$-algebra).
Let us fix $x\in\mathbf{X}$; we will set $X$ to $x$, $X:=x$.
For a given value $y\in\mathbf{Y}$, define
\begin{equation}\label{eq:p}
  p_y
  :=
  \sum_{z\in\mathbf{Z}}
  P(Z=z) P(Y=y\mid X=x,Z=z).
\end{equation}
The interpretation of \eqref{eq:p}
is that it is the probability of $Y=y$
in the mutilated causal model
in which the arrow from $Z$ to $X$ in Figure~\ref{fig:XYZ} has been removed
and $X$ has been set to $x$.
We do not mention $x$, which is fixed, in our notation.

Generate an IID random sample $(X_n,Y_n,Z_n)$, $n\in[N]$,
of size $N\ge1$ from $P$.
For each $y\in\mathbf{Y}$ define the estimate $F_y$ of $p_y$ by
\begin{equation}\label{eq:F}
  F_y
  :=
  \sum_{z\in\mathbf{Z}}
  \frac{\left|\{n\in[N]:Z_n=z\}\right|+1}{N+1}
  \times
  \frac{\left|\{n\in[N]:(X_n,Y_n,Z_n)=(x,y,z)\}\right|+1}{\left|\{n\in[N]:(X_n,Z_n)=(x,z)\}\right|+1}.
\end{equation}

\begin{lemma}\label{lem:main}
  For each $y\in\mathbf{Y}$, it is true that
  \begin{equation}\label{eq:lem-E}
    \E
    \frac{p_y}{F_y}
    \le
    1.
  \end{equation}
\end{lemma}

We will prove Lemma~\ref{lem:main} in Appendix~\ref{app:proofs}.

Let $Y_{N+1}$ be a random variable,
independent of the sample $(X_n,Y_n,Z_n)$, $n\in[N]$,
with values in $\mathbf{Y}$
taking each value $y\in\mathbf{Y}$ with probability $p_y$ given by \eqref{eq:p}.
In causal inference we will use the following corollary
(see \cite{Ramdas/Wang:2025} for hypothesis testing using e-variables).

\begin{corollary}\label{cor:main}
  For each probability measure $Q$ on $\mathbf{Y}$,
  the random variable $E$ defined by
  \begin{equation}\label{eq:cor-E}
    E
    :=
    \frac{Q(\{Y_{N+1}\})}{F_{Y_{N+1}}}
  \end{equation}
  is an e-variable (i.e., is nonnegative with expectation at most 1).
\end{corollary}

\begin{proof}
  The statement of the corollary follows from
  \[
    \E
    \frac{Q(\{Y_{N+1}\})}{F_{Y_{N+1}}}
    =
    \sum_{y\in\mathbf{Y}}
    p_y
    \E
    \frac{Q(\{y\})}{F_y}
    =
    \sum_{y\in\mathbf{Y}}
    Q(\{y\})
    \E
    \frac{p_y}{F_y}
    \le
    \sum_{y\in\mathbf{Y}}
    Q(\{y\})
    =
    1,
  \]
  where the first equality follows from the independence of $Y_{N+1}$.
\end{proof}

Two particularly natural choices for $Q$ in Corollary~\ref{cor:main}
are the uniform probability measure on $\mathbf{Y}$
and the probability measure concentrated on some $y^*\in\mathbf{Y}$.
The former choice treats all labels $y\in\mathbf{Y}$ symmetrically.
The latter choice is appropriate when a specific label $y^*$
is much more important for us than the other labels
and we would like to exclude it confidently when possible;
an example of such a label is  ``patient's death''.

Let us apply Corollary~\ref{cor:main} to a simple problem of causal inference.
After generating the IID random sample $(X_n,Y_n,Z_n)$, $n\in[N]$,
we generate $Y_{N+1}$ according to the probability measure
assigning probability $p_y$ to each $y\in\mathbf{Y}$
(conditionally on the random sample, of course).
This is exactly the probability distribution on $Y$ in the mutilated causal network
obtained by setting $X:=x$.

Corollary~\ref{cor:main} allows us to output e-prediction regions for $Y_{n+1}$.
Given an alternative $Q$ (say, the uniform probability measure on $\mathbf{Y}$)
and a significance level $\alpha>0$ (typically a large number, such as 10 or 100),
we define the corresponding e-prediction region by
\begin{equation}\label{eq:region}
  \Gamma^{\alpha}
  :=
  \left\{
    y\in\mathbf{Y}:
    \frac{Q(\{y\})}{F_y}<\alpha
  \right\}.
\end{equation}
The prediction regions $\Gamma^{\alpha}$ are nested and grow as $\alpha$ increases.
The main property of validity for this predictor is
\begin{equation}\label{eq:validity}
  \int_0^{\infty}
  \P
  \left(
    Y \notin \Gamma^{\alpha}
  \right)
  \dd\alpha
  \le
  1;
\end{equation}
in words, the probability of error at significance level $\alpha>0$
should integrate to at most 1
\cite[the end of Appendix B]{Vovk:2025PR}.
Since the probability of error decreases in $\alpha$,
the probability of error at level $\alpha$ does not exceed $1/\alpha$
by Markov's inequality.
This is a simpler property of validity;
however, the property of validity expressed by \eqref{eq:validity}
is much stronger
(and in fact is equivalent to the $E$ defined by \eqref{eq:cor-E}
being an e-variable).

For a large $N$ and small $\left|\mathbf{Z}\right|$,
the e-prediction regions \eqref{eq:region} are close to being optimal in some sense.
Namely, the ``oracle'' e-prediction regions are
\begin{equation}\label{eq:ideal}
  \Gamma^{\alpha}
  :=
  \left\{
    y\in\mathbf{Y}:
    \frac{Q(\{y\})}{p_y}<\alpha
  \right\}
\end{equation}
if $Q$ is a genuine alternative \cite[Chap.~3]{Ramdas/Wang:2025}.
Since $F_y$ as defined by \eqref{eq:F} is an estimate of $p_y$,
\eqref{eq:region} is an approximation to \eqref{eq:ideal}
that can be computed from the data.

One special case mentioned earlier is where $Q$ is concentrated
on some outcome $y^*$, which we would like to avoid (``death of the patient'').
For a large $\alpha$, we are justified in predicting $Y_n\ne y^*$
when we observe $F_{y^*}\le1/\alpha$.
In particular, the probability that we will make an error
does not exceed $1/\alpha$ (where the probability is marginal,
not conditional on $F_{y^*}\le1/\alpha$).

\begin{remark}\label{rem:back-door}
  For simplicity, in this note
  we concentrate on the simplest causal graph given in Figure~\ref{fig:XYZ}.
  However, it is easy to extend our approach
  to any causal graph covered by the popular back-door criterion:
  see \cite[Theorem 3.3.2]{Pearl:2009}.
  Now $Z$ becomes a set of variables in general,
  called an adjustment set.
\end{remark}

\section{No stable stochastic mechanism for $X$}
\label{sec:game}

As pointed out in \cite[Sect.~1]{Vovk:1996Gam} (see also \cite[Sect.~3.6.1]{Pearl:2009}),
the assumption that $X_n$, $n\in[N]$, are output by a stable stochastic mechanism
is unnatural in the context of causal inference,
since we are contemplating setting $X$ to some value $x$.
In this section we will still assume that $Z_n$ and $Y_n$ are generated
by stable stochastic mechanisms but will drop this assumption for $X_n$.

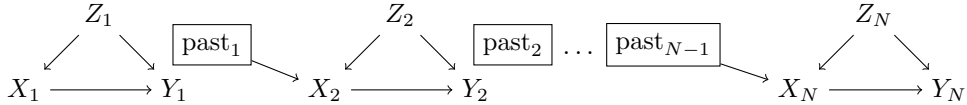
\begin{figure}
  \centerline{%
    \begin{tikzpicture}
      \node (X1) at (0,0) {$X_1$};
      \node (Y1) at (2,0) {$Y_1$};
      \node (Z1) at (1,1) {$Z_1$};
      \draw [->] (Z1) edge (X1);
      \draw [->] (X1) edge (Y1);
      \draw [->] (Z1) edge (Y1);
      \node (X2) at (4,0) {$X_2$};
      \node (Y2) at (6,0) {$Y_2$};
      \node (Z2) at (5,1) {$Z_2$};
      \draw [->] (Z2) edge (X2);
      \draw [->] (X2) edge (Y2);
      \draw [->] (Z2) edge (Y2);
      \node (XN) at (10.3,0) {$X_N$};
      \node (YN) at (12.3,0) {$Y_N$};
      \node (ZN) at (11.3,1) {$Z_N$};
      \draw [->] (ZN) edge (XN);
      \draw [->] (XN) edge (YN);
      \draw [->] (ZN) edge (YN);
      \node[draw, rectangle] (past1) at (2.5,0.6) {$\text{past}_1$};
      \draw [->] (past1) edge (X2);
      \node[draw, rectangle] (past2) at (6.5,0.6) {$\text{past}_2$};
      \node at (7.4,0.5) {$\dots$}; 
      \node[draw, rectangle] (past3) at (8.5,0.6) {$\text{past}_{N-1}$};
      \draw [->] (past3) edge (XN);
    \end{tikzpicture}}
  \caption{The repeated causal graph}
  \label{fig:XYZ-repeated}
\end{figure}

Our assumptions are represented graphically in Figure~\ref{fig:XYZ-repeated},
which essentially consists of a series of causal triangles
analogous to the one in Figure~\ref{fig:XYZ}.
The triangles are arranged chronologically from left to right
(while the order inside the triangles is not chronological, of course;
within each triangle the arrows represent the causal order).
The boxes labelled $\text{past}_n$, $n\in[N-1]$, stand for sets of variables.
In this and next sections
we will consider two possible interpretations of these boxes;
in any case, $\text{past}_n$ represents some variables in the first $n$ triangles.

The most satisfying analogue of Lemma~\ref{lem:main}
obtains when each box $\text{past}_n$ stands
for all the variables $X_i$ and $Z_i$, $i\in[n]$.
We will call this the \emph{$Y$-oblivious interpretation}
of Figure \ref{fig:XYZ-repeated}.
Under this interpretation, each $X_{n+1}$, $n\in[N-1]$,
has incoming arrows from all $X_i$ and $Z_i$, $i\in[n]$.

\begin{lemma}\label{lem:oblivious}
  Lemma~\ref{lem:main} continues to hold under the $Y$-oblivious interpretation
  of Figure \ref{fig:XYZ-repeated}.
\end{lemma}

The proof of Lemma~\ref{lem:oblivious} is given in Appendix~\ref{app:proofs}.
Corollary~\ref{cor:main} also holds in the $Y$-oblivious setting,
and therefore, we still have e-prediction regions~\eqref{eq:region}
satisfying the property of validity~\eqref{eq:validity}.

\section{Conclusion}
\label{sec:conclusion}

This note extends conformal e-prediction to a simple setting of causal inference.
(And in Appendix~\ref{app:proofs} we will see that our causal e-predictor
is a combination of $\left|\mathbf{Z}\right|$ conformal e-predictors.)
In this section we discuss some open questions and directions of further research.

In this note we only discussed causal inference
in the simplest situation of Figure~\ref{fig:XYZ},
but in Remark~\ref{rem:back-door}
we noticed that the extension to the back-door criterion
is straightforward.
There is no doubt that our method can be extended
to other cases in which Pearl's do calculus \cite{Pearl:2009}
and its variations allow us to identify the causal effect of $X$ on $Y$.
In the case of the back-door criterion,
our method works best for an adjustment set $Z$
that has the smallest product of the sizes of its elements' domains
(assuming $Z$ is not unique).

Our proof of Lemma~\ref{lem:oblivious} does not work
for the \emph{strong interpretation} of Figure~\ref{fig:XYZ-repeated},
where each $\text{past}_n$ stands for all the previous variables,
namely $(X_i,Y_i,Z_i)$, $i\in[n]$.
One possible approach in this situation is
to use conformal test martingales
\cite[Chaps.~8 and~9]{Vovk/etal:2022book}.

We assumed that each variable takes only finitely many values,
but another natural setting is regression,
in which $Y$ is allowed to be any real number
and we are interested in prediction intervals for it.

We have not discussed the optimality of our finite-sample results.
Our simulation studies suggest that \eqref{eq:lem-E} will hold
even if we use less heavy regularization in \eqref{eq:F}
(e.g., replacing the entries of ``${}+1$'' by ``${}+c$'' for $c<1$).
What are the admissible constants in \eqref{eq:lem-E}?
Improving them will lead to an automatic improvement
in the e-prediction regions \eqref{eq:region}.
A slack in \eqref{eq:lem-E} is also discussed in Sect.~\ref{subsec:slack}.

\appendix
\section{Connections with conformal prediction}

This note is motivated by conformal e-prediction,
which is applicable
to the case where we have only one random variable $Y$
and to the case where we have two random variables $X,Y$.
Both cases will be used in the proofs of Lemmas~\ref{lem:main} and~\ref{lem:oblivious}.

\subsection{Simple conformal e-prediction}

First suppose that we have only one random variable $Y$.
Let $P$ be a positive probability measure on $\mathbf{Y}$
generating $Y$.
Generate an IID random sample $Y_n$, $n\in[N+1]$, from $P$
of size $N+1$.
For a fixed $y\in\mathbf{Y}$, consider the nonconformity scores
\[
  \alpha_i
  :=
  \frac{N+1}{\left|\{n\in[N+1]:Y_n=y\}\right|}
  1_{\{Y_i=y\}},
  \quad
  i\in[N+1]
\]
\cite[Sect.~2]{Vovk:2025PR}.
The standard property of validity for conformal e-prediction gives
\begin{equation*}
  \E
  \left(
    \frac{N+1}{\left|\{n\in[N+1]:Y_n=y\}\right|}
    1_{\{Y_{N+1}=y\}}
  \right)
  \le
  1,
\end{equation*}
i.e.,
\begin{equation}\label{eq:Y}
  \E
  \frac{N+1}{\left|\{n\in[N]:Y_n=y\}\right|+1}
  \le
  \frac{1}{P(Y=y)}.
\end{equation}

It is instructive to see what Corollary~\ref{cor:main}
becomes in the case of simple conformal e-prediction.
We can rewrite \eqref{eq:Y} as
\begin{equation*}
  \E
  \frac{P(Y=y)}{\hat P(Y=y)}
  \le
  1,
\end{equation*}
where
\[
  \hat P(Y=y)
  :=
  \frac{\left|\{n\in[N]:Y_n=y\}\right|+1}{N+1}
\]
(so that $\hat P$ is an estimate of $P$,
although it is a ``super-probability measure''
rather than a probability measure
in that $\hat P(Y=y)$ sum to more than 1 over $y\in\mathbf{Y}$).
The same argument as in the proof of Corollary~\ref{cor:main}
shows that,
for any probability measure $Q$ on $\mathbf{Y}$,
\begin{equation*}
  \E
  \frac{Q(\{Y_{N+1}\})}{\hat P(\{Y_{N+1}\})}
  \le
  1.
\end{equation*}
In particular, the conformal e-prediction set
at significance level $\alpha>0$ (typically $\alpha\gg1$) is
\[
  \left\{
    y\in\mathbf{Y}:
    Q(\{y\})/\hat P(\{y\}) < \alpha
  \right\}.
\]

\subsection{Conditional conformal e-prediction}
\label{subsec:conditional}

Now suppose that we have a positive probability measure
$P$ on $\mathbf{X}\times\mathbf{Y}$
generating random variables $X$ and $Y$.
We generate an IID random sample $(X_n,Y_n)$, $n\in[N+1]$, from it.
For a given $y\in\mathbf{Y}$,
we use the nonconformity scores
\[
  \alpha_i
  :=
  \frac
    {\left|\{n\in[N+1]:X_n=X_i\}\right|}
    {\left|\{n\in[N+1]:(X_n,Y_n)=(X_i,y)\}\right|}
  1_{\{Y_i=y\}},
  \quad
  i\in[N+1],
\]
in the object-conditional conformal e-predictor
\cite[Sect.~4]{Vovk:2025PR} ($0/0$ is interpreted as 1 here).
Now we have, for a fixed $x\in\mathbf{X}$,
\begin{multline*}
  \E
  \biggl(
    \frac
      {\left|\{n\in[N+1]:X_n=x\}\right|}
      {\left|\{n\in[N+1]:(X_n,Y_n)=(x,y)\}\right|}
    1_{\{(X_{N+1},Y_{N+1})=(x,y)\}}\\
  \biggm|
    X_1,\dots,X_{N+1}
  \biggr)
  \le
  1,
\end{multline*}
i.e.,
\begin{equation}\label{eq:XY}
  \E
  \left(
    \frac
      {\left|\{n\in[N]:X_n=x\}\right|+1}
      {\left|\{n\in[N]:(X_n,Y_n)=(x,y)\}\right|+1}
  \middle|
    X_1,\dots,X_{N}
  \right)
  \le
  \frac{1}{P(Y=y\mid X=x)}.
\end{equation}

It is clear that for the validity of conditional conformal e-prediction
it suffices to assume that the $Y_n$ are IID conditional on $X_n$;
the distribution of $X_n$ can be arbitrary
\cite[Sect.~4.6.1]{Vovk/etal:2022book}.

\subsection{Slack in (\ref{eq:lem-E})}
\label{subsec:slack}

This section continues discussion started in Sect.~\ref{sec:conclusion}.
The presence of a slack in \eqref{eq:lem-E}
becomes transparent if we assume that $P(X=x)=1$.
Then \eqref{eq:F} can be rewritten as
\begin{align*}
  F
  &=
  \sum_{z\in\mathbf{Z}}
  \frac{\left|\{n\in[N]:Z_n=z\}\right|+1}{N+1}
  \times
  \frac{\left|\{n\in[N]:(Y_n,Z_n)=(y,z)\}\right|+1}{\left|\{n\in[N]:Z_n=z\}\right|+1} \\
  &=
  \sum_{z\in\mathbf{Z}}
  \frac{\left|\{n\in[N]:(Y_n,Z_n)=(y,z)\}\right|+1}{N+1}
  = \frac{\left|\{n\in[N]:Y_n=y\}\right|+\left|\mathbf{Z}\right|}{N+1},
\end{align*}
and so \eqref{eq:lem-E} is weaker than what can be obtained
with conformal e-prediction,
which allows us to have $1$ in place of the $\left|\mathbf{Z}\right|$.
Namely, when $P(X=x)=1$, \eqref{eq:lem-E} becomes
\begin{equation*}
  \E
  \left(
    P(Y=y)
    \middle/
    \frac{\left|\{n\in[N]:Y_n=y\}\right|+\left|\mathbf{Z}\right|}{N+1}
  \right)
  \le
  1
\end{equation*}
and so is weaker than \eqref{eq:Y}.

\section{Some proofs}
\label{app:proofs}

\begin{proof}[Proof of Lemma~\ref{lem:main}]
  We start from noticing that $p_y/F_y$
  is the weighted harmonic mean over $z\in\mathbf{Z}$ of $p_{y,z}/F_{y,z}$
  taken with the weights $p_{y,z}/p_y$,
  where
  \begin{equation*}
    p_{y,z}
    :=
    P(Z=z) P(Y=y\mid Z=z,X=x)
  \end{equation*}
  and
  \begin{equation*}
    F_{y,z}
    :=
    \frac{\left|\{n\in[N]:Z_n=z\}\right|+1}{N+1}
    \times
    \frac{\left|\{n\in[N]:(X_n,Y_n,Z_n)=(x,y,z)\}\right|+1}{\left|\{n\in[N]:(X_n,Z_n)=(x,z)\}\right|+1}.
  \end{equation*}
  This follows from the equalities $F_y=\sum_z F_{y,z}$ and $p_y=\sum_z p_{y,z}$,
  which imply
  \[
    \frac{p_y}{F_y}
    =
    \frac{1}{\sum_{z\in\mathbf{Z}}F_{y,z}/p_y}
    =
    \frac{1}{\sum_{z\in\mathbf{Z}}\frac{p_{y,z}}{p_y}\frac{F_{y,z}}{p_{y,z}}}.
  \]
  Since a weighted harmonic mean does not exceed
  the corresponding weighted arithmetic mean,
  it suffices to prove $\E(p_{y,z}/F_{y,z})\le1$
  for a fixed $z\in\mathbf{Z}$,
  i.e.,
  \begin{multline}\label{eq:Z1}
    \E
    \left(
      \frac{N+1}{\left|\{n\in[N]:Z_n=z\}\right|+1}
      \times
      \frac{\left|\{n\in[N]:(X_n,Z_n)=(x,z)\}\right|+1}{\left|\{n\in[N]:(X_n,Y_n,Z_n)=(x,y,z)\}\right|+1}
    \right)\\
    \le
    \frac{1}{P(Z=z)P(Y=y\mid X=x,Z=z)}.
  \end{multline}
  We can deduce \eqref{eq:Z1} from the property of validity for conformal e-prediction
  \begin{equation}\label{eq:basic}
    \E
    \frac{N+1}{\left|\{n\in[N]:Z_n=z\}\right|+1}
    \le
    \frac{1}{P(Z=z)}
  \end{equation}
  (see \eqref{eq:Y})
  and its version
  \begin{equation}\label{eq:basic-1}
    \E
    \frac{\left|\{n\in[N]:(X_n,Z_n)=(x,z)\}\right|+1}{\left|\{n\in[N]:(X_n,Y_n,Z_n)=(x,y,z)\}\right|+1}
    \le
    \frac{1}{P(Y=y\mid X=x,Z=z)}
  \end{equation}
  that holds conditionally on $X_1,\dots,X_N,Z_1,\dots,Z_N$
  (see \eqref{eq:XY}, where $X$ should be replaced by $(X,Z)$).
  Indeed, \eqref{eq:Z1} can be obtained
  by multiplying \eqref{eq:basic} and \eqref{eq:basic-1}.
\end{proof}

\begin{proof}[Proof of Lemma~\ref{lem:oblivious}]
  As in the previous proof, it suffices to prove \eqref{eq:Z1}.
  We have \eqref{eq:basic},
  and we also have \eqref{eq:basic-1} conditionally
  on $(X_n,Z_n)$, $n\in[N]$
  (which follows from the remark at the end of Sect.~\ref{subsec:conditional}).
  Combining \eqref{eq:basic} and \eqref{eq:basic-1} gives \eqref{eq:Z1}.
\end{proof}
\end{document}